\documentclass[11pt]{article}
\usepackage{amssymb}

\usepackage{verbatim}
\usepackage{latexsym,amsfonts}
\usepackage{amsmath}
\usepackage{times}
\usepackage{xcolor}
\usepackage{lipsum}

\newcommand*{\qed}{\hfill\ensuremath{\blacksquare}}

\newcommand*{\qedc}{\hfill\ensuremath{\Box}}

\newcommand\blfootnote[1]{%
	\begingroup
	\renewcommand\thefootnote{}\footnote{#1}%
	\addtocounter{footnote}{-1}%
	\endgroup
}

\newtheorem{theorem}{Theorem}[section]

\newtheorem{lemma}[theorem]{Lemma}

\newtheorem{remark}[theorem]{Remark}

\usepackage{lipsum} 
\newenvironment{proof}[1][Proof]{\noindent\textbf{#1.} }{\ \rule{0.5em}{0.5em}}
\begin{document}

\title{On existence and nonexistence of isoperimetric inequality with differents monomial weights}

\author{Emerson Abreu\thanks{Email: eabreu@ufmg.br}\ \ \&\ Leandro G. Fernandes Jr\thanks{Email: leandrogonzaga@ufmg.br}\\
Departamento de Matem\'atica\\
Universidade Federal de Minas Gerais\\
\bigskip
CP702 - 30123-970 Belo Horizonte-MG, Brazil }
        \date{\vspace{-5ex}}
\maketitle

\begin{abstract}
We consider the monomial weight $x^{A}=\vert x_{1}\vert^{a_{1}}\ldots\vert x_{N}\vert^{a_{N}}$, where $a_{i}$ is a nonnegative real number for each $i\in\{1,\ldots,N\}$, and we establish the existence and nonexistence of isoperimetric inequalities with different monomial weights. We study positive minimizers of $\int_{\partial\Omega}x^{A}\mathcal{H}^{N-1}(x)$ among all smooth bounded open sets $\Omega$ in $\mathbb{R}^{N}$ with fixed Lebesgue measure  with monomial weight $\int_{\Omega}x^{B}dx$.

\end{abstract}

\section{Introduction and main results}
\label{intro}

\blfootnote{This study was financed by the Coordena\c c\~ao de Aperfei\c coamento de Pessoal de N\'{\i}vel Superior - Brasil (CAPES) and Funda\c c\~ao de Apoio \`a Pesquisa do Estado de Minas Gerais - (FAPEMIG)}

A great attention has been given recently to the isoperimetric inequalities with weights, see for instance \cite{Alvino}, \cite{Alvino2}, \cite{Alvino3}, \cite{Betta}, \cite{Betta1}, \cite{Boyer}, \cite{Brock}, \cite{Cabre1}, \cite{Cabre2}, \cite{Cabre-Ros-Oton},  \cite{CARROLL}, \cite{Csato}, \cite{Csato1}, \cite{Dahlberg}, \cite{Di Giosia}, \cite{Fusco}, \cite{Rosales}
and the references therein. However, in the wide literature, most works approach volume functional and perimeter functional carrying the same weight.

It is worth emphasizing that some researchers have been studying isoperimetric inequalities when the volume and perimeter carry two different weights, see \cite{Alvino}, \cite{Alvino2}, and \cite{Csato1}. In \cite{Alvino}, motivated by some norm inequalities with weights which are well-known as Caffarelli-Kohn-Niremberg (see \cite{Caffarelli}), it was studied by Alvino et al., the following isoperimetric inequality:
\begin{align}\label{0.0.0.1.0}
	\text{minimize} \, \ \int_{\partial\Omega}\vert x\vert^{k}\mathcal{H}^{N-1}(x) \, \ \text{among all smooth sets} \, \ \Omega\subset\mathbb{R}^{N} \, \ \text{satisfying} \, \ \int_{\Omega} \vert x\vert^{l}dx=1.
	\end{align} 
	
The existence of an isoperimetric inequality with monomial weights was shown by  Cabr\'e, and Ros-Oton, see Theorem 1.4 in \cite{Cabre-Ros-Oton}, namely

\textbf{Theorem A (Cabr\'e-Ros-Oton)} Let $A=(a_{1},\ldots, a_{N})$ be a nonnegative vector in $\mathbb{R}^{N}$, $x^{A}=\vert x_{1}\vert^{a_{1}}\ldots\vert x_{N}\vert^{a_{N}}$,\linebreak $D=a_{1}+\cdots+a_{N}+N$, and $\mathbb{R}_{A}^{N}=\left\{(x_{1},\ldots, x_{N}); x_{i}>0 \, \ \text{whenever}\, \ a_{i}>0\right\}$. Let $\Omega\subset\mathbb{R}^{N}$ be a bounded Lischitz domain. Denote
\begin{align*}
	m(\Omega)=\int_{\Omega} x^{A}dx \, \, \text{and} \, \, P(\Omega)=\int_{\partial\Omega}x^{A}d\mathcal{H}^{N-1}(x).
\end{align*}
Then, \begin{align}\label{0.0.0.1}
	\dfrac{P(\Omega)}{m(\Omega)^{\frac{D-1}{D}}}\geq\dfrac{P(B_{1}^{A})}{m(B_{1}^{A})^{\frac{D-1}{D}}},
\end{align}
where $B_{1}^{A}:=B_{1}(0)\cap \mathbb{R}_{A}^{N}$. 

As in the classical case, the inequality (\ref{0.0.0.1}) implies the following Sobolev Inequality with monomial weights
\begin{align}\label{0.0.0.2}
	\left(\int_{\mathbb{R}_{A}^{N}}\vert u\vert^{p^{\star}}x^{A}dx\right)^{\frac{1}{p^{\star}}}\leq C_{p,N}	\left(\int_{\mathbb{R}_{A}^{N}}\vert \nabla u\vert^{p}x^{A}dx\right)^{\frac{1}{p}},
\end{align}
for every $u\in C_{c}^{1}(\Omega)$, where $p^{\star}=\frac{pD}{D-p}$, and $p<D$. The best constant in (\ref{0.0.0.2}) is given by

\begin{align*}
	C_{1}=D\left(\frac{\Gamma\left(\frac{a_{1}+1}{2}\right)\cdots\Gamma\left(\frac{a_{N}+1}{2}\right)}{2^{k}\Gamma\left(1+\frac{D}{2}\right)}\right)^{\frac{1}{D}}\, \  \text{for}\, \ p=1
\end{align*}

and by

 \begin{align*}
	C_{p, N}=C_{1}D^{\frac{1}{D}-1-\frac{1}{p}}\left(\frac{p-1}{D-p}\right)^{\frac{1}{p'}}\left(\frac{p'\Gamma(D)}{\Gamma\left(\frac{D}{p}\right)\Gamma\left(\frac{D}{p'}\right)}\right)^{\frac{1}{D}}, \, \ \text{for} \, \ 1<p<D,
\end{align*} 
where $p=\frac{p}{p-1}$, and $k$ is the number of strictly positive entries of $A$.

Additionally, the best constant $C_{p,N}$ gives the possibility to prove a Trudinger-Moser type inequality, more especially, that there exists constants $c_{1}>0$ and $c_{2}>0$ such that
\begin{align*}
\int_{\Omega}\exp{\left[\left(\frac{c_{1}\vert u(x)\vert}{\Vert\nabla u\Vert_{L^{D}(\Omega, x^{A}dx)}}\right)^{\frac{D}{D-1}}\right]}x^{A}dx\leq c_{2}\int_{\Omega}x^{A}dx	
\end{align*}
where $\Omega\subset\mathbb{R}^{N}$ is a bounded open set.

Motivated by inequality (\ref{0.0.0.2}) and the Caffarelli-Kohn-Niremberg inequality, Castro presented in \cite{Castro} the following result

\textbf{Theorem B (Castro)} Consider $N\geq 1$, $p\geq 1$, $F=(f_{1},\ldots, f_{N})$, $G=(g_{1},\ldots, g_{N})\in \mathbb{R}^{N}$. Let $f=f_{1}+\cdots+f_{N}$ and $g=g_{1}+\cdots+g_{N}$, for $p^{\ast}\geq 1$ defined by
\begin{align*}
	\dfrac{1}{p^{\ast}}+\dfrac{g+1}{N}=\dfrac{1}{p}+\dfrac{f}{N},
\end{align*}
suppose \begin{enumerate}
\item [1.] $\dfrac{1}{p^{\ast}}f_{i}+\left(1-\dfrac{1}{p}\right)g_{i}>0$ for all $i=1,\ldots,N$,
\item [2.] $0\leq f_{i}-g_{i}<1$ for all $i=1,\ldots,N$.
\item [3.] $1-\dfrac{N}{p}<f-g\leq 1$.
\end{enumerate}
Then there exists a constant $C>0$ such that for all $u\in C_{c}^{1}\left(\mathbb{R}^{N}\right)$
\begin{align*}
	\left(\int_{\mathbb{R}^{N}}\vert x^{G} u(x)\vert^{p^{\ast}}dx\right)^{\frac{1}{p^{\ast}}}\leq C\left(\int_{\mathbb{R}^{N}}\vert x^{F}\nabla u(x)\vert^{p}\right)^{\frac{1}{p}}.
\end{align*}

For $p=1$, we may rewrite the previous result as:

\textit{The following three conditions 
\begin{enumerate}
	\item [i)]$a_{i}>0$,
	\item [ii)] $0\leq a_{i}-\frac{N+a-1}{N+b}b_{i}<1$,
	\item [iii)] $a-b\leq 1$.
\end{enumerate}
are sufficient for the existence of a constant $C>0$, that depends only on $a$, $b$, and $N$, such that
\begin{align*}
	\left(\int_{\mathbb{R}^{N}} x^{B}\vert u(x)\vert^{\frac{N+b}{N+a-1}}dx\right)^{\frac{N+a-1}{N+b}}\leq C\int_{\mathbb{R}^{N}}x^{A}\vert\nabla u(x)\vert dx,
\end{align*}
for every $u\in C_{c}^{1}\left(\mathbb{R}^{N}\right)$.}

Motivated by Theorem B and problem $(\ref{0.0.0.1.0})$, we approach the existence and nonexistence of isoperimetric inequality where the volume and perimeter have different monomial weights, more especific, we study the following isoperimetric problem:

Find the constant $C_{A,B,N}\in [0,+\infty)$, where \begin{align}\label{0.0.0.8}
C_{A,B,N}:=\inf\left\{\dfrac{\displaystyle\int_{\partial\Omega}x^{A}d\mathcal{H}^{N-1}(x)}{\left[\displaystyle\int_{\Omega}x^{B}dx\right]^{\frac{N+a-1}{N+b}}}; \Omega \, \ \text{is a smooth open set and} \, \ 0<\int_{\Omega}x^{B}dx<\infty\right\}.
\end{align}  

Even though some cases in one dimension are included, throughout the paper we consider $N\geq2$.  For the case $N=1$ see \cite{Alvino2}. One of our main results is:

\begin{theorem}\label{0.0.0.9} Consider $N\geq 2$. Let $A=(a_{1},\ldots,a_{N})$, $B=(b_{1},\ldots, b_{N})\in\mathbb{R}^{N}$ be two nonnegative vectors. Consider $a=a_{1}+\cdots+a_{N}$, $b=b_{1}+\cdots+b_{N}$, $\overline{a}_{i}=a-a_{i}$, and $\overline{b}_{i}=b-b_{i}$. Then we have the following
\begin{enumerate}	
\item[$(I)$] if \begin{align*} C_{A,B,N}>0, \end{align*} 
 then 
 
 \begin{align}\label{0.0.0.9.0}
 0\leq a_{i}-\frac{N+a-1}{N+b}b_{i}\leq\frac{N+a-1}{N+b}
 \end{align}
 or equivalently  
  \begin{align}\label{0.0.0.9.1} 0\leq a_{i}-\frac{N+\overline{a}_{i}-1}{N+\overline{b}_{i}}b_{i}\, \, \  \text{and} \, \, \  \frac{a_{i}}{b_{i}+1}\leq\frac{N+\overline{a}_{i}-1}{N+\overline{b}_{i}-1}.
 \end{align}
 \item [$(II)$] if $a-b\leq 1$ and the condition $(\ref{0.0.0.9.0})$ holds, then
 \begin{align*}
 	C_{A,B,N}>0.
 \end{align*} 
\end{enumerate}

\end{theorem}

For the case $a-b=1$, on certain conditions, we present the exactly value of $C_{A,B, N}$.
\begin{theorem}\label{0.0.0.10} Consider $N\geq 2$.
	Let $A=(a_{1},\ldots,a_{N})$, $B=(b_{1},\ldots, b_{N})\in\mathbb{R}^{N}$ be two nonnegative vectors. Consider $a=a_{1}+\cdots+a_{N}$, $a=b_{1}+\cdots+b_{N}$, $\overline{a}_{i}=a-a_{i}$, and $\overline{b}_{i}=b-b_{i}$. If $a_{j}=b_{j}$ for all $j\in\{1,\ldots, N\}\backslash\{i\}$, and $a_{i}=b_{i}+1$, then
	\begin{align*}
	C_{A,B, N}=a_{i}.
	\end{align*} 
\end{theorem}

Our Theorem \ref{0.0.0.9} establishes all cases of existence and nonexistence of isoperimetric inequality for two nonnegative vectors satisfying $a-b\leq1$, which also implies the improvement and the necessity of $(ii)$ in the Theorem A. The condition $(\ref{0.0.0.9.1})$, equivalent to $(\ref{0.0.0.9.0})$, is even more general, because it shows us how to choose the entrie $i$ of the vectors $A$ and $B$, since we have already chosen the others $N-1$ entries. For instance, if we have $N-1$ entries iguals  in the vectors $A$ and $B$, $a_{j}=b_{j}$ for all $j\in\{1,\ldots, N\}\backslash\{i\}$, then the condition $(\ref{0.0.0.9.1})$ tells us that the isoperimetric inequality exists only if $a_{i}\leq b_{i}+1$.

The Theorem \ref{0.0.0.10} is surprising, since $C_{A,B,N}$ in this case does not depend on $N$. It is worth emphasizing that in the proof we get a decreasing sequence $\left(\Omega_{\varepsilon}\right)_{\varepsilon>0}\subset\mathbb{R}^{N}$, it means $\Omega_{\varepsilon}\subset\Omega_{\delta}$ whenever $\varepsilon<\delta$, such that 
\begin{align*}
	\dfrac{\int_{\partial\Omega_{\varepsilon}}x^{A}\mathcal{H}^{N-1}(x)}{\int_{\Omega_{\varepsilon}}x^{B}dx}\to a_{i} \, \ \text{as} \, \ \varepsilon\to 0,
\end{align*} 
however the $\int_{\Omega_{\varepsilon}}x^{A}dx\to 0$ as $\varepsilon\to 0$.	

The paper is organized as follows. In section 2, we define some basic elements that we will use throughout the paper. In section 3, we state some lemmata which will be used in the prove of Theorem \ref{0.0.0.9}. Finally, in section 4, we prove the Theorem \ref{0.0.0.10}.

{\section{Some definitions}}

Let us introduce some elements that we will use in this paper.

Given a nonnegative function $\omega:\mathbb{R}^{N}\to\mathbb{R}$, locally lipschitz on  $\mathbb{R}^{N}$,  we set the $P_{\omega}$-Perimeter of a measurable set $M$ by 

\begin{align*}
P_{\omega}(M):=\sup\left\{\int_{M}div(\omega(x)\nu(x))dx; \nu\in C_{0}^{1}(\mathbb{R}^{N},\mathbb{R}^{N}), \vert\nu\vert\leq 1\, \ \text{on}\, \ \mathbb{R}^{N}\right\}.
\end{align*}
When we consider the specific density $\omega(x)=x^{A}:=\vert x_{1}\vert^{a_{1}}\cdot\ldots\cdot\vert x_{N}\vert^{a_{N}}$, we denote $P_{A}$, instead of $P_{x^{A}}$.

If $\Omega$ is a smooth bounded open set, then the weighted perimeter is equivalent to the following 
\begin{align*}
	P_{\omega}(\Omega) = \int_{\partial\Omega}\omega(x) d\mathcal{H}^{N-1}(x),
\end{align*}
here $\mathcal{H}^{N-1}$ is the $(N-1)$-dimensional Hausdorff measure.
 
For a nonnegative measurable function $\gamma:\mathbb{R}^{N}\to\mathbb{R}$, we set by $m_{\gamma}$ the Lebesgue measure with weight $\gamma(x)dx$, namely,

\begin{align*}
	m_{\gamma}(M)=\int_{M}\gamma(x)dx,
\end{align*}
where $M$ is a $\mathcal{H}^{N}$-measurable set. Similarly, if $\gamma(x)=x^{B}:=\vert x\vert^{b_{1}}\cdot\ldots\cdot\vert x_{N}\vert^{b_{N}}$, we denote $m_{B}$, instead of $m_{x^{B}}$.

We now consider a measurable set $M$ with $0<m_{\gamma}(M)<\infty$, and we define
\begin{align*}
\mathcal{R}_{A,B, N}(M):=\dfrac{P_{A}(M)}{\left[m_{B}(M)\right]^{\frac{N+a-1}{N+b}}}.
\end{align*} For $\Omega\subset\mathbb{R}^{N}$ a smooth bounded open set, we then have

\begin{align*}
\mathcal{R}_{A, B, N}(\Omega):=\dfrac{\displaystyle\int_{\partial\Omega}x^{A}d\mathcal{H}^{N-1}(x)}{\left[\displaystyle\int_{\Omega}x^{B}dx\right]^{\frac{N+a-1}{N+b}}}.
\end{align*}

It is worth emphasizing that the constant $C_{A,B,N}$ (defined in (\ref{0.0.0.8})) satisfies
\begin{align*}
C_{A,B,N}=\inf\left\{\mathcal{R}_{A,B,N}(M);M \, \ \text{is measurable with} \, \ 0<m_{B}(M)<+\infty\right\}.
\end{align*}	 

 We also set \begin{align*}
	\mathcal{Q}_{A,B,N}(u):=\dfrac{\displaystyle\int_{\mathbb{R}^{N}}\vert\nabla u(x)\vert x^{A}dx}{\left[\displaystyle\int_{\mathbb{R}^{N}}\vert u\vert^{\frac{N+b}{N+a-1}}x^{B}dx\right]^{\frac{N+a-1}{N+b}}},
\end{align*}
for every  $u\in C_{c}^{1}(\mathbb{R}^{N})\backslash\{0\}$. Besides that, throughout this paper we will use the following notation: 

We say that a vector $A\in\mathbb{R}^{N}$ is nonnegative if all its entries are nonnegative. 

For $x=(x_{1},\ldots,x_{i-1}, x_{i}, x_{i+1},\ldots,x_{k-1}, x_{k},x_{k+1}, \ldots, x_{n})\in\mathbb{R}^{N}$ a vector, and

  $A=(a_{1},\ldots,a_{i-1},a_{i}, a_{i+1},\ldots,a_{k-1}, a_{k},a_{k+1}, \ldots, a_{n})\in\mathbb{R}^{N}$ 
   a nonnegative vector, we denote by

\begin{enumerate} 
\item[] $\overline{x}_{i}:=(x_{1},\ldots,x_{i-1},x_{i+1},\ldots, x_{n})$; 
\item[] $\overline{A}_{i}:=(a_{1},\ldots,a_{i-1},a_{i+1},\ldots, a_{n})$;
 \item[] $\overline{x}_{ik}:=(x_{1},\ldots,x_{i-1},x_{i+1},\ldots, \ldots,x_{k-1},x_{k+1}, \ldots x_{n})$; 
\item []
$\overline{A}_{ik}:=(a_{1},\ldots,a_{i-1},a_{i+1},\ldots, \ldots,a_{k-1},a_{k+1}, \ldots a_{n})$; 
\item [] $\overline{a}_{i}:=a-a_{i}=a_{1}+\cdots+a_{i-1}+a_{i+1}+\cdots+a_{n}$;
\item [] $\overline{a}_{ik}:=a-a_{i}-a_{k}=a_{1}+\cdots+a_{i-1}+a_{i+1}+\cdots a_{k-1}+a_{k+1}\cdots+a_{n}$.
\end{enumerate}
Finally, when $N\in\mathbb{N}$ and $r>0$, we denote by $B_{N}(r)$ the ball centered in $0$ and radius $r$ in $\mathbb{R}^{N}$, moreover $B_{N}^{+}(r)=B_{N}(r)\cap\mathbb{R}_{+}^{N}$, where $\mathbb{R}_{+}^{N}:=\{x=(x_{1},\ldots,x_{N})\in\mathbb{R}^{N}; x_{i}>0\, \, \text{for every} \, \ i\in \{1,\ldots, N\}\}$, and $\mathbb{R}^{N}_{A}=\{x=(x_{1},\ldots, x_{N})\in\mathbb{R}^{N};x_{i}>0\, \, \text{whenever}\, \ a_{i}>0\}$.
{\section{proof of the Theorem \ref{0.0.0.9}}}

This section contains relevant results for the two theorems presented in the introduction. Here, we prove the item $(i)$ of Theorem \ref{0.0.0.9} based on two important lemmata, moreover we estabilish the sufficient condition $(ii)$ using classical arguments such as coarea formula.

Borrowing ideas from \cite{Alvino}, we establish the following important result.
\begin{lemma}\label{3.0.1.0} Let $A=(a_{1},\ldots, a_{N})$ and $B=(b_{1},\ldots, a_{N})$ be two nonnegative vectors in $\mathbb{R}^{N}$.
 If \begin{align*}
 	C_{A,B,N}>0
 \end{align*}
then \begin{align*}a_{i}-\frac{N+a-1}{b+N}b_{i}\geq 0\end{align*}
or equivalently \begin{align*}
	a_{i}-\frac{N+\overline{a}_{i}-1}{N+\overline{b}_{i}}b_{i}\geq 0
\end{align*}
\end{lemma}

\begin{proof}  Arguing by contradiction, we assume that \begin{align}\label{3.0.1.1}a_{i}-\frac{N+a-1}{b+N}b_{i}<0.\end{align} 
 Consider $t>2$ and $B(te_{i},1)$ the ball centered in $te_{i}$ and radius $1$. 

Using the area formula, we obtain
 \begin{align}\label{3.0.1.2}
	\int_{\partial B(te_{i},1)}x^{A}d\mathcal{H}^{N-1}(x)&=\int_{x_{1}^{2}+\cdots +x_{i-1}^{2}+(x_{i}-t)^{2}+x_{i+1}^{2}+\dots+x_{N}^{2}=1}\vert x_{1}\vert^{a_{1}}\cdot\ldots\cdot\vert x_{N}\vert^{a_{N}} d\mathcal{H}^{N-1}(x)\nonumber\\
	&=\int_{B_{N-1}(1)}\left\vert t+(1-\vert\overline{x}_{i}\vert^{2})^{\frac{1}{2}}\right\vert^{a_{i}}\frac{\overline{x}_{i}^{\overline{A}_{i}}}{(1-\vert \overline{x}_{i}\vert^{2})^{\frac{1}{2}}}d\overline{x}_{i}\nonumber\\
	&+\int_{B_{N-1}(1)}\left\vert t-(1-\vert\overline{x}_{i}\vert^{2})^{\frac{1}{2}}\right\vert^{a_{i}}\frac{\overline{x}_{i}^{\overline{A}_{i}}}{(1-\vert\overline{x}_{i}\vert^{2})^{\frac{1}{2}}}d\overline{x}_{i}\nonumber\\
	&\leq(1+2^{a_{i}})t^{a_{i}}\int_{B_{N-1}(1)}\frac{\overline{x}_{i}^{\overline{A}_{i}}}{(1-\vert\overline{x}_{i}\vert^{2})^{\frac{1}{2}}}d\overline{x}_{i}	
\end{align}

On the other hand, by change of variable and elementary inequalities, we get

\begin{align}\label{3.0.1.3}
\int_{ B(te_{i},1)}x^{B}dx&=\int_{x_{1}^{2}+\cdots +x_{i-1}^{2}+(x_{i}-t)^{2}+x_{i+1}^{2}+\cdots+x_{N}^{2}<1}\vert x_{1}\vert^{b_{1}}\cdot\cdots\cdot\vert x_{N}\vert^{b_{N}} dx\nonumber\\
&=\int_{t-1}^{t+1}\vert x_{i}\vert^{b_{i}}\left(\int_{B_{N-1}\left(\left[1-(x_{i}-t)^{2}\right]^{\frac{1}{2}}\right)}\overline{x}_{i}^{\overline{B}_{i}}d\overline{x}_{i}\right)dx_{i}\nonumber\\
&=\int_{t-1}^{t+1}\vert x_{i}\vert^{b_{i}}\left(1-(x_{i}-t)^{2}\right)^{\frac{\overline{b}_{i}+(N-1)}{2}}dx_{i}\int_{B_{N-1}(1)}\overline{x}_{i}^{\overline{B}_{i}}d\overline{x}_{i}\nonumber\\
&=\int_{B_{N-1}(1)}\overline{x}_{i}^{\overline{B}_{i}}d\overline{x}_{i}\int_{-1}^{1}\vert y+t\vert^{b_{i}}\left(1-y^{2}\right)^{\frac{\overline{b}_{i}+(N-1)}{2}}dy\nonumber\\
&\geq\int_{B_{N-1}(1)}\overline{x}_{i}^{\overline{B}_{i}}d\overline{x}_{i}\int_{0}^{1}\vert y+t\vert^{b_{i}}\left(1-y^{2}\right)^{\frac{\overline{b}_{i}+(N-1)}{2}}dy\nonumber\\
&\geq t^{b_{i}}\displaystyle\int_{B_{N-1}(1)}\overline{x}_{i}^{\overline{B}_{i}}d\overline{x}_{i}\int_{0}^{1}\left(1-y^{2}\right)^{\frac{\overline{b}_{i}+(N-1)}{2}}dy.
\end{align}
It follows from inequalities $(\ref{3.0.1.2})$ and $(\ref{3.0.1.3})$ that 
\begin{align}\label{3.0.1.4}
\dfrac{	\displaystyle\int_{\partial B(te_{i},1)}x^{A}d\mathcal{H}^{N-1}(x)}{\left[\displaystyle\int_{ B(te_{i},1)}x^{B}dx\right]^{\frac{N+a-1}{N+b}}}\leq\dfrac{(1+2^{a_{i}})t^{a_{i}}\displaystyle\int_{B_{N-1}(1)}\frac{\overline{x}_{i}^{\overline{A}_{i}}}{(1-\vert\overline{x}_{i}\vert^{2})^{\frac{1}{2}}}d\overline{x}_{i}}{\left[t^{b_{i}}\displaystyle\int_{B_{N-1}(1)}\overline{x}_{i}^{\overline{B}_{i}}d\overline{x}_{i}\int_{0}^{1}\left(1-y^{2}\right)^{\frac{\overline{b}_{i}+(N-1)}{2}}\right]^{\frac{N+a-1}{N+b}}}
\end{align}
Thus by $(\ref{3.0.1.1})$ and inequality $(\ref{3.0.1.4})$, we obtain \begin{align*}
	\displaystyle\lim_{t\to\infty}\dfrac{	\displaystyle\int_{\partial B(te_{i},1)}x^{A}d\mathcal{H}^{N-1}(x)}{\left[\displaystyle\int_{ B(te_{i},1)}x^{A}dx\right]^{\frac{N+a-1}{N+b}}}=0.
\end{align*}
Which is a contradiction with $	C_{A,B,N}>0$.
\end{proof}

The previous Lemma gives us the first behavior and huge dependence of the vector $B=(b_{1},\ldots,b_{N})$ with respect to the vector $A=(a_{1},\ldots, a_{N})$. For instance, if $a_{i}=0$, then the isoperimetric inequality exists only if $b_{i}=0$. 

\begin{lemma}\label{3.0.2.0}
	Let $A=(a_{1},\ldots, a_{N})$ and $B=(b_{1},\ldots, a_{N})$ be two nonnegative vectors in $\mathbb{R}^{N}$. If
	\begin{align*}
	C_{A,B,N}>0
	\end{align*}
	then \begin{align*}
	a_{i}-\dfrac{N+a-1}{N+b}b_{i}\leq\dfrac{N+a-1}{N+b}
	\end{align*}
or equivalently \begin{align*}  \frac{a_{i}}{b_{i}+1}\leq\frac{N+\overline{a}_{i}-1}{N+\overline{b}_{i}-1}.
\end{align*}
\end{lemma}

\begin{proof} Again, by an argument of contradiction, we assume that \begin{align}\label{3.0.2.0.1}
	a_{i}-\dfrac{N+a-1}{N+b}b_{i}>\dfrac{N+a-1}{N+b}.
	\end{align} 
We define for a positive $\varepsilon$ the set

\begin{align*}
	\Omega_{\varepsilon}=\left\{x\in\mathbb{R}^{N};\vert x\vert< R^{2}, x_{j}>0\, \ \text{for all}\, \ j\in\{1,\ldots, N\}\, \ \text{and}\, \ x_{i}<\varepsilon\vert \overline{x}_{i}\vert \right\}.
  \end{align*}
 We may see that
  \begin{align}\label{3.0.2.1}
 	&\partial\Omega_{\varepsilon}=\left\{x\in\mathbb{R}^{N};x_{j}>0\, \ \text{for all}\, \ j\in\{1,\ldots, N\}, x_{i}=\varepsilon\vert\overline{x}_{i}\vert, \, \ \text{and}\, \ \vert\overline{x}_{i}\vert\leq \dfrac{R}{(1+\varepsilon^{2})^{\frac{1}{2}}} \right\}\nonumber\\
 	&\bigcup\left\{x\in\mathbb{R}^{N};x_{j}>0\, \ \text{for all}\, \ j\in\{1,\ldots, N\},\dfrac{R}{(1+\varepsilon^{2})^{\frac{1}{2}}}\leq\vert\overline{x}_{i}\vert\leq R, \, \ \text{and} \, \ x_{i}=\left(R^{2}-\vert\overline{x}_{i}\vert^{2}\right)^{\frac{1}{2}}\right\}\nonumber\\
 	&\bigcup\left\{x\in\mathbb{R}^{N};x_{j}>0\, \ \text{for all}\, \ j\in\{1,\ldots, N\}\backslash\{i\}, x_{i}=0, \vert x\vert\leq R\right\}\nonumber\\
 	&\bigcup_{k=1, k\neq i}^{N}\left\{x\in\mathbb{R}^{N};x_{j}>0\, \ \text{for all}\, \ j\in\{1,\ldots, N\}\backslash\{k\}, x_{k}=0, \vert x\vert\leq R, \, \ \text{and}\, \ x_{i}\leq\varepsilon\vert\overline{x}_{ik}\vert\right\}\nonumber\\
 	&=:A_{\varepsilon}^{1}\cup A_{\varepsilon}^{2}\cup A_{\varepsilon}^{3}\bigcup_{k=1, k\neq i}^{N} C_{\varepsilon}^{k}.
 \end{align}
 By definition of $\Omega_{\varepsilon}$ and change of variable, we get
 \begin{align}\label{3.0.2.2}
 \int_{\Omega_{\varepsilon}}x^{B}dx&=\int_{B_{N-1}^{+}\left(\frac{R}{(1+\varepsilon^{2})^{\frac{1}{2}}}\right)}\int_{0}^{\varepsilon\vert\overline{x}_{i}\vert }\overline{x}_{i}^{\overline{B}_{i}}x_{i}^{b_{i}}dx_{i}d\overline{x}_{i}+\int_{B_{N-1}^{+}(R)\backslash B_{N-1}^{+}\left(\frac{R}{(1+\varepsilon^{2})^{1/2}}\right)}\int_{0}^{\left(R^{2}-\vert \overline{x}_{i}\vert^{2}\right)^{1/2}}\overline{x}_{i}^{\overline{B}_{i}}x_{i}^{b_{i}}dx_{i}d\overline{x}_{i}\nonumber\\
 &\geq\dfrac{\varepsilon^{b_{i}+1}}{b_{i}+1}\int_{B_{N-1}^{+}\left(\frac{R}{(1+\varepsilon^{2})^{\frac{1}{2}}}\right)}\overline{x}_{i}^{\overline{B}_{i}}\vert\overline{x}_{i}\vert^{b_{i}+1}d\overline{x}_{i}\nonumber\\
 &=\dfrac{\varepsilon^{b_{i}+1}R^{N+b}}{(b_{i}+1)(1+\varepsilon^{2})^{\frac{N+b}{2}}}\int_{B_{N-1}^{+}(1)}\overline{x}_{i}^{\overline{B}_{i}}\vert\overline{x}_{i}\vert^{b_{i}+1}d\overline{x}_{i}.
 \end{align}

By (\ref{3.0.2.1}), we obtain

\begin{align}\label{3.0.2.3}
\int_{\partial\Omega_{\varepsilon}} x^{A}d\mathcal{H}^{N-1}(x)&=\int_{A_{\varepsilon}^{1}}x^{A}d\mathcal{H}^{N-1}(x)+\int_{A_{\varepsilon}^{2}}x^{A}d\mathcal{H}^{N-1}(x)+\int_{A_{\varepsilon}^{3}}x^{A}d\mathcal{H}^{N-1}(x)\nonumber\\
&+\sum_{k=1, k\neq i}^{N}\int_{C_{\varepsilon}^{k}}	x^{A}d\mathcal{H}^{N-1}(x).
\end{align}
We now estimate the boundary area with density $x^{A}d\mathcal{H}^{N-1}(x)$. First, we calculate on $C_{\varepsilon}^{k}$'s.

Let $k\neq i$. If $a_{k}>0$, then
\begin{align}\label{3.0.2.4}
	\int_{C_{\varepsilon}^{k}}x^{A}d\mathcal{H}^{N-1}(x)=0.
\end{align} 
Otherwise, if $a_{k}=0$, then
\begin{align}\label{3.0.2.5}
\int_{C_{\varepsilon}^{k}}x^{A}d\mathcal{H}^{N-1}(x)&=\int_{B_{N-2}^{+}\left(\frac{R}{(1+\varepsilon^{2})^{1/2}}\right)}\int_{0}^{\varepsilon\vert \overline{x}_{ik}\vert}\overline{x}_{ik}^{\overline{A}_{ik}}x_{i}^{a_{i}}dx_{i}d\overline{x}_{ik}\nonumber\\
&+\int_{B_{N-2}^{+}(R)\backslash B_{N-2}^{+}\left(\frac{R}{(1+\varepsilon^{2})^{1/2}}\right)}\int_{0}^{\left(R^{2}-\vert \overline{x}_{ik}\vert^{2}\right)^{1/2}}\overline{x}_{ik}^{\overline{A}_{ik}}x_{i}^{a_{i}}dx_{i}d\overline{x}_{ik}\nonumber\\
&=\dfrac{\varepsilon^{a_{i}+1}}{a_{i}+1}\int_{B_{N-2}^{+}\left(\frac{R}{(1+\varepsilon^{2})^{1/2}}\right)}\overline{x}_{ik}^{\overline{A}_{ik}}\vert \overline{x}_{ik}\vert^{a_{i}+1}d\overline{x}_{ik}\nonumber\\
&+\frac{1}{a_{i}+1}\int_{B_{N-2}^{+}(R)\backslash B_{N-2}^{+}\left(\frac{R}{(1+\varepsilon^{2})^{1/2}}\right)}\overline{x}_{ik}^{\overline{A}_{ik}}\left(R^{2}-\vert\overline{x}_{ik}\vert^{2}\right)^{\frac{a_{i}+1}{2}}d\overline{x}_{ik} \nonumber\\
&=\dfrac{\varepsilon^{a_{i}+1}R^{N+a-1}}{(a_{i}+1)(1+\varepsilon^{2})^{\frac{N+a-1}{2}}}\int_{B_{N-2}^{+}(1)}\overline{x}_{ik}^{\overline{A}_{ik}}\vert \overline{x}_{ik}\vert^{a_{i}+1}d\overline{x}_{ik}\nonumber\\
&+\dfrac{R^{N+a-1}}{(a_{i}+1)\left(1+\varepsilon^{2}\right)^{\frac{N+\overline{a}_{i}-2}{2}}}\int_{B_{N-2}^{+}\left((1+\varepsilon^{2})^{\frac{1}{2}}\right)\backslash B_{N-2}^{+}\left(1\right)}\left(1-\dfrac{\vert \overline{x}_{ik}\vert^{2}}{1+\varepsilon^{2}}\right)^{\frac{a_{i}+1}{2}}\overline{x}_{ik}^{\overline{A}_{ik}}d\overline{x}_{ik}\nonumber\\
&\leq \dfrac{R^{N+a-1}O(\varepsilon^{a_{i}+1})}{\left(1+\varepsilon^{2}\right)^{\frac{N+a-1}{2}}}+\dfrac{R^{N+a-1}\varepsilon^{a_{i}+1}}{\left(1+\varepsilon^{2}\right)^\frac{N+a-1}{2}}\int_{B_{N-2}^{+}\left((1+\varepsilon^{2})^{\frac{1}{2}}\right)\backslash B_{N-2}^{+}\left(1\right)}\overline{x}_{ik}^{\overline{A}_{ik}}d\overline{x}_{ik}\nonumber\\
&\leq \dfrac{R^{N+a-1}O(\varepsilon^{a_{i}+1})}{\left(1+\varepsilon^{2}\right)^{\frac{N+a-1}{2}}}+\dfrac{R^{N+a-1}\varepsilon^{a_{i}+1}}{\left(1+\varepsilon^{2}\right)^\frac{N+a-1}{2}}\left(\left(1+\varepsilon^{2}\right)^{\frac{N+\overline{a}_{i}-2}{2}}-1\right)\int_{ B_{N-2}^{+}\left(1\right)}\overline{x}_{ik}^{\overline{A}_{ik}}d\overline{x}_{ik}\nonumber\\
&\leq \dfrac{R^{N+a-1}O(\varepsilon^{a_{i}+1})}{\left(1+\varepsilon^{2}\right)^{\frac{N+a-1}{2}}}+\dfrac{R^{N+a-1}O(\varepsilon^{a_{i}+3})}{\left(1+\varepsilon^{2}\right)^\frac{N+a-1}{2}}.
\end{align}

We now compute the boundary area on $A_{\varepsilon}^{1}$. It follows from Area Formula and change of variable that

\begin{align}\label{3.0.2.6}
\int_{A_{\varepsilon}^{1}}x^{A}d\mathcal{H}^{N-1}(x)&=\int_{B_{N-1}^{+}\left(\frac{R}{(1+\varepsilon^{2})^{\frac{1}{2}}}\right)}\overline{x}_{i}^{\overline{A}_{i}}\varepsilon^{a_{i}}\vert \overline{x}_{i}\vert^{a_{i}}(1+\varepsilon^{2})^{\frac{1}{2}}d\overline{x}_{i}\nonumber\\
&=\dfrac{\varepsilon^{a_{i}}R^{N+a-1}}{(1+\varepsilon^{2})^{\frac{N+a-2}{2}}}\int_{B_{N-1}^{+}(1)}\overline{x}_{i}^{\overline{A}_{i}}\vert\overline{x}_{i}\vert^{a_{i}}d\overline{x}_{i}.
\end{align} 

Finally, we estimate the last integral. By change of variable and elementary inequalities, we obtain 
\begin{align}\label{3.0.2.7}
\int_{A_{\varepsilon}^{2}}x^{A}d\mathcal{H}^{N-1}(x)&=\int_{B_{N-1}^{+}(R)\backslash B_{N-1}^{+}\left(\frac{R}{(1+\varepsilon^{2})^{1/2}}\right)}\overline{x}_{i}^{\overline{A}_{i}}\left(R^{2}-\vert\overline{x}_{i}
\vert^{2}\right)^{\frac{a_{i}}{2}}d\overline{x}_{i}\nonumber\\
&=R^{N+a-1}\int_{B_{N-1}^{+}(1)\backslash B_{N-1}^{+}\left(\frac{1}{(1+\varepsilon^{2})^{1/2}}\right)}\overline{x}_{i}^{\overline{A}_{i}}\left(1-\vert\overline{x}_{i}\vert^{2}\right)^{\frac{a_{i}}{2}}d\overline{x}_{i}\nonumber\\
&=\dfrac{R^{N+a-1}}{(1+\varepsilon^{2})^{\frac{N+\overline{a}_{i}-1}{2}}}\int_{B_{N-1}^{+}((1+\varepsilon^{2})^{\frac{1}{2}})\backslash B_{N-1}^{+}\left(1\right)}\overline{x}_{i}^{\overline{A}_{i}}\left(1-\dfrac{\vert \overline{x}_{i}\vert^{2}}{1+\varepsilon^{2}}\right)^{\frac{a_{i}}{2}}d\overline{x}_{i}\nonumber\\
&\leq\dfrac{R^{N+a-1}}{(1+\varepsilon^{2})^{\frac{N+\overline{a}_{i}-1}{2}}}\int_{B_{N-1}^{+}((1+\varepsilon^{2})^{\frac{1}{2}})\backslash B_{N-1}^{+}\left(1\right)}\overline{x}_{i}^{\overline{A}_{i}}\left(1-\dfrac{1}{1+\varepsilon^{2}}\right)^{\frac{a_{i}}{2}}d\overline{x}_{i}\nonumber\\
&=\dfrac{R^{N+a-1}\varepsilon^{a_{i}}}{(1+\varepsilon^{2})^{\frac{N+a-1}{2}}}\int_{B_{N-1}^{+}((1+\varepsilon^{2})^{\frac{1}{2}})\backslash B_{N-1}^{+}\left(1\right)}\overline{x}_{i}^{\overline{A}_{i}}d\overline{x}_{i}\nonumber\\
&=\dfrac{R^{N+a-1}\varepsilon^{a_{i}}}{(1+\varepsilon^{2})^{\frac{N+a-1}{2}}}\left((1+\varepsilon^{2})^{\frac{N+\overline{a}_{i}-1}{2}}-1\right)\int_{ B_{N-1}^{+}\left(1\right)}\overline{x}_{i}^{\overline{A}_{i}}d\overline{x}_{i}\nonumber\\
&=R^{N+a-1}O(\varepsilon^{a_{i}+2}).
\end{align}

Thus, it follows from $(\ref{3.0.2.2})$, $(\ref{3.0.2.3})$, $(\ref{3.0.2.4})$ or $(\ref{3.0.2.5})$, $(\ref{3.0.2.6})$, and $(\ref{3.0.2.7})$ that
\begin{align}\label{3.0.2.8}
\dfrac{P_{A}(\Omega_{\varepsilon})}{\left[m_{B}(\Omega_{\varepsilon})\right]^{\frac{N+a-1}{N+b}}}& \leq\dfrac{\dfrac{\varepsilon^{a_{i}}R^{N+a-1}}{(1+\varepsilon^{2})^{\frac{N+a-2}{2}}}\displaystyle\int_{B_{N-1}^{+}(1)}\overline{x}_{i}^{\overline{A}_{i}}\vert \overline{x}_{i}\vert^{a_{i}}d\overline{x}_{i}+R^{N+a-1}\left(O(\varepsilon^{a_{i}+1})+O(\varepsilon^{a_{i}+2})+O(\varepsilon^{a_{i}+3})\right)}{\left[\dfrac{\varepsilon^{b_{i}+1}R^{N+b}}{(b_{i}+1)(1+\varepsilon^{2})^{\frac{N+b}{2}}}\displaystyle\int_{B_{N-1}^{+}(1)}\overline{x}_{i}^{\overline{B}_{i}}\vert\overline{x}_{i}\vert^{b_{i}+1}d\overline{x}_{i}\right]^{\frac{N+a-1}{N+b}}}\nonumber\\
&=\varepsilon^{a_{i}-\frac{N+a-1}{N+b}(b_{i}+1)}\dfrac{\frac{1}{(1+\varepsilon^{2})^{\frac{N+a-2}{2}}}\displaystyle\int_{B_{N-1}^{+}(1)}\overline{x}_{i}^{\overline{A}_{i}}\vert \overline{x}_{i}\vert^{a_{i}}d\overline{x}_{i}}{\left[\dfrac{1}{(b_{i}+1)(1+\varepsilon^{2})^{\frac{N+b}{2}}}\displaystyle\int_{B_{N-1}^{+}(1)}\overline{x}_{i}^{\overline{B}_{i}}\vert \overline{x}_{i}\vert^{b_{i}+1}d\overline{x}_{i}\right]^{\frac{N+a-1}{N+b}}}\nonumber\\
&+O\left(\varepsilon^{a_{i}+1-\frac{N+a-1}{N+b}(b_{i}+1)}\right)+O\left(\varepsilon^{a_{i}+2-\frac{N+a-1}{N+b}(b_{i}+1)}\right)+O\left(\varepsilon^{a_{i}+3-\frac{N+a-1}{N+b}(b_{i}+1)}\right)
\end{align}
Therefore, the inequality $(\ref{3.0.2.8})$, and $(\ref{3.0.2.0.1})$ imply that \begin{align*}
	\lim_{\varepsilon\to 0}\dfrac{P_{A}(\Omega_{\varepsilon})}{\left[m_{B}(\Omega_{\varepsilon})\right]^{\frac{N+a-1}{N+b}}}=0.
\end{align*}
Which is a contradiction with our assumption.
\end{proof}


The next result is expected and the proof relies on classical arguments, see for example \cite{Talenti1}. For convenience of the reader, we sketch the proof. 

\begin{lemma}\label{3.0.3.0}
	Let $\Omega$ be a Lipschitz bounded open set. Consider $\omega$ a nonnegative locally lipschitz function and $\gamma$ a nonnegative continuous function on $\mathbb{R}^{N}$.  Then there exists a smooth and compactly supported sequence $(u_{\varepsilon})_{\varepsilon>0}$ on $\mathbb{R}^{N}$ such that
	\begin{align}\label{3.0.3.1}
	\lim_{\varepsilon\to 0}\int_{\mathbb{R}^{N}}\vert u_{\varepsilon}\vert^{p} \gamma(x)dx=\int_{\Omega}\gamma(x)dx,\, \ \text{for each}\, \ p\geq 1,
	\end{align}
and mainly \begin{align}\label{3.0.3.2}
\lim_{\varepsilon\to 0}\int_{\mathbb{R}^{N}}\vert \nabla u_{\varepsilon}(x)\vert\omega(x)dx=\int_{\partial\Omega}\omega(x)dx.
\end{align}
\end{lemma}

\begin{proof}
We begin with the following assertion.

\textbf{Claim 1}. \begin{align*}
\int_{\mathbb{R}^{N}}\left\vert\chi_{\Omega}(x+h)-\chi_{\Omega}(x)\right\vert dx\leq\vert h\vert\mathcal{H}^{N-1}(\partial\Omega),
\end{align*}
where $\chi_{\Omega}$ is the characteristic function on the set $\Omega$, and $h$ is any vector in $\mathbb{R}^{N}$. 

\textit{proof of the claim 1.} Let $\varphi$ be a smooth and compactly supported function on $\mathbb{R}^{N}$. We then have
\begin{align*}
	\int_{\mathbb{R}^{N}}\left[\chi_{\Omega}(x+h)-\chi_{\Omega}(x)\right]\varphi(x)dx=\int_{\mathbb{R}^{N}}\chi_{\Omega}(x)\left[\varphi(x-h)-\varphi(x)\right]dx=\int_{\Omega}\left[\varphi(x-h)-\varphi(x)\right]dx.	
\end{align*} 
By fundamental theorem of calculus and divergent theorem, we get
\begin{align*}
	\int_{\Omega}\varphi(x-h)-\varphi(x)dx&=-\int_{\Omega}\int_{0}^{1}\nabla\varphi(x-th)h dtdx\\
	&=-\int_{\Omega}\left(h\int_{0}^{1}\nabla\varphi(x-th)dt\right)dx\\
	&=-\int_{\partial\Omega}\left(\int_{0}^{1}\varphi(x-th)dt\right)\langle h, \eta(x)\rangle\mathcal{H}^{N-1}(x),
\end{align*}
where $\eta$ denotes the outward unit normal vector with respect to $\Omega$.

This gives the estimate,
\begin{align*}
	\left\vert \int_{\mathbb{R}^{N}}\left[\chi_{\Omega}(x+h)-\chi_{\Omega}(x)\right]\varphi(x)dx\right\vert\leq\sup_{y\in\mathbb{R}^{N}}\vert\varphi(y)\vert\vert h\vert\mathcal{H}^{N-1}(\partial\Omega).
	\end{align*}
Thus, the proof of claim 1 follows. \qedc

\textbf{Claim 2.} Let a mollifier $\rho\in C_{c}^{\infty}(\mathbb{R}^{N})$ supported in the unit ball $B_{N}(0,1)$. We define 
\begin{align*}
u_{\varepsilon}(x):=\rho_{\varepsilon}\ast\chi_{\Omega}(x)=\int_{\mathbb{R}^{N}}\rho_{\varepsilon}(x-y)\chi_{\Omega}(y)dy,
\end{align*}
where $\rho_{\varepsilon}(x)=\varepsilon^{-N}\rho\left(\frac{x}{\varepsilon}\right)$. Then \begin{align*}
	u_{\varepsilon}\to\chi_{\Omega} \, \ \text{in} \, \ L^{1}(\Omega, dx), \,\text{and}\, \  L^{1}(\Omega, \gamma(x) dx).
\end{align*}
\textit{proof of the claim 2.} 
By properties of the function $\rho$, we obtain
\begin{align*}
u_{\varepsilon}(x)-\chi_{\Omega}(x)=\int_{\mathbb{R}^{N}}\rho_{\varepsilon}(y)\left[\chi_{\Omega}(x-y)-\chi_{\Omega}(x)\right]dy.
\end{align*}
By the previous inequality and claim 1, it follows that
\begin{align*}
	\int_{\mathbb{R}^{N}}\vert u_{\varepsilon}(x)-\chi_{\Omega}(x)\vert\gamma(x)dx&\leq C(\gamma,\Omega)\int_{\mathbb{R}^{N}}\vert u_{\varepsilon}(x)-\chi_{\Omega}(x)\vert dx\\
	&\leq C(\gamma,\Omega)\int_{\Omega}\rho_{\varepsilon}(y)\int_{\mathbb{R}^{N}}\left\vert\chi_{\Omega}(x-y)-\chi_{\Omega}(x)\right\vert dxdy\\
	&\leq C(\gamma,\Omega)\mathcal{H}^{N-1}(\partial\Omega)\int_{\mathbb{R}^{N}} \vert y\vert\rho_{\varepsilon}(y)dy\\
	&=\varepsilon C(\gamma,\Omega)\mathcal{H}^{N-1}(\partial\Omega)\int_{\mathbb{R}^{N}}\vert y\vert\rho(y)dy,
\end{align*}
where $C(\gamma,\Omega)=\sup\{\gamma(y); y\in\mathbb{R}^{N}, dist(y,\Omega)<1\}$. 

Thus, the claim 2 follows, and so the equality $(\ref{3.0.3.1})$. \qedc 

Now, we concern on the equality $(\ref{3.0.3.2})$. Taking $f\in C_{c}^{1}(\mathbb{R}^{N};\mathbb{R}^{N})$, we get 
\begin{align}\label{3.0.3.4}
	\int_{\mathbb{R}^{N}}u_{\varepsilon}(x)div\left(w(x)f(x)\right)dx=-\int_{\mathbb{R}^{N}}\langle\nabla u_{\varepsilon}(x),\omega(x)f(x)\rangle dx.
\end{align}
We then have
\begin{align*}
	\left\vert\int_{\mathbb{R}^{N}}u_{\varepsilon}(x)div\left(\omega(x)f(x)\right)dx\right\vert\leq\sup_{y\in\mathbb{R}^{N}}\vert f(y)\vert\int_{\mathbb{R}^{N}}\vert\nabla u_{\varepsilon}(x)\vert\omega(x)dx.
\end{align*}
 Taking the supremum over all $f\in C_{c}^{1}(\mathbb{R}^{N};\mathbb{R}^{N})$ satistying $\vert f\vert\leq 1$ on $\mathbb{R}^{N}$, we get
 \begin{align}\label{3.0.3.6}
 \int_{\partial\Omega}\omega(x)d\mathcal{H}^{N-1}\leq\liminf_{\varepsilon\to 0}\int_{\mathbb{R}^{N}}\vert\nabla u_{\varepsilon}(x)\vert\omega(x)dx.
 \end{align}   
  
For the proof of the reverse inequality, we consider $\delta>0$ arbitrary. By uniform continuity of $\omega$ on $\partial\Omega$,  there exists $\theta(\delta,\partial\Omega)>0$, that depends only on $\delta$ and $\partial\Omega$, such that
\begin{align*}
	\left\vert\omega(x+y)-\omega(x)\right\vert<\delta
\end{align*}
whenever $\vert y\vert<\theta(\delta,\partial\Omega)$.

It follows from equality $(\ref{3.0.3.4})$, divergence theorem and previous statement that
\begin{align}\label{3.0.3.7}
		\left\vert\int_{\mathbb{R}^{N}}\langle\nabla u_{\varepsilon}(x),f(x)\rangle \omega(x)dx\right\vert&=\left\vert\int_{\mathbb{R}^{N}}u_{\varepsilon}(x)div\left(w(x)f(x)\right)dx\right\vert\nonumber\\
		&=\left\vert\int_{\mathbb{R}^{N}}\rho_{\varepsilon}(y)\int_{\Omega}div\left(\omega(x+y)f(x+y)\right)dxdy\right\vert\nonumber\\
		&=\left\vert\int_{\mathbb{R}^{N}}\rho_{\varepsilon}(y)\int_{\partial\Omega}\langle f(x+y), \eta(x)\rangle\omega(x+y)d\mathcal{H}^{N-1}(x)dy\right\vert\nonumber\\
		&=\left\vert\int_{\mathbb{R}^{N}}\rho_{\varepsilon}(y)\int_{\partial\Omega}\langle f(x+y), \eta(x)\rangle\omega(x+y)d\mathcal{H}^{N-1}(x)dy\right\vert\nonumber\\
		&\leq\sup_{y\in\mathbb{R}^{N}}\vert f(y)\vert\left[\int_{\mathbb{R}^{N}}\rho_{\varepsilon}(y)\int_{\partial\Omega}\vert\omega(x+y)-w(x)\vert d\mathcal{H}^{N-1}(x)dy\right.\nonumber\\
		&+\left.\int_{\mathbb{R}^{N}}\rho_{\varepsilon}(y)\int_{\partial\Omega}\omega(x)d\mathcal{H}^{N-1}(x)dy\right]\nonumber\\
		&\leq\sup_{y\in\mathbb{R}^{N}}\vert f(y)\vert\left[\delta\mathcal{H}^{N-1}(\partial\Omega)+\int_{\partial\Omega}\omega(x)d\mathcal{H}^{N-1}(x)\right].
\end{align}
Here, $\eta$ denotes the outward unit normal vector with respect to $\Omega$, and $\varepsilon<\theta(\delta,\partial\Omega)$.

Applying the reverse Hölder inequality to the inequality (\ref{3.0.3.7}), we obtain
\begin{align}\label{3.0.3.8}
	\int_{\mathbb{R}^{N}}\vert\nabla u_{\varepsilon}(x)\vert\omega(x)dx\leq\delta\mathcal{H}^{N-1}(\partial\Omega)+\int_{\partial\Omega}\omega(x)\mathcal{H}^{N-1}(x), \, \ \text{for every} \, \  \varepsilon<\theta(\delta,\partial\Omega).
\end{align}
By inequalities $(\ref{3.0.3.6})$, and $(\ref{3.0.3.8})$, we get the equality $(\ref{3.0.3.2})$, and the proof of the lemma is complete.
\end{proof}

\begin{remark}\label{3.0.3.1.0} Given a Lipschitz bounded open set, in order to analyze the isoperimetric quotient
	\begin{align}\label{3.0.3.1.1}
		\dfrac{\displaystyle\int_{\partial\Omega}x^{A}d\mathcal{H}^{N-1}(x)}{\left[\displaystyle\int_{\Omega}x^{B}dx\right]^{\frac{N+a-1}{N+b}}},
	\end{align}
it is sufficient to consider $\Omega$ contained in $\mathbb{R}_{A}^{N}$, if $a-b\leq 1$. The strategy below is due to Cabr\'e and Ros-Oton, see \cite{Cabre-Ros-Oton}.
	
 We may assume, by symmetry, that $A=(a_{1}, \ldots, a_{k}, 0, \ldots, 0)$, where $a_{i}>0$ for every $i\in\{1,\ldots, k\}$ and  some $0\leq k\leq N$. We split the domain $\Omega$ in at most $2^{k}$ disjoint subdomains $\Omega_{j}$, $j\in\{1,\ldots,J \}$, where each subdomain $\Omega_{j}$ is contained in the cone $\{\varepsilon_{i}x_{i}>0, i\in\{1,\ldots, k\}\}$ for different $\varepsilon_{i}\in\{-1,1\}$. Thus, we have $\overline{\Omega}=\overline{\Omega}_{1}\cup\ldots\cup\overline{\Omega}_{J}$, 
\begin{align*}
P_{A}(\Omega)=\sum_{j=1}^{J}P_{A}(\Omega_{j}), \, \   \text{since the weight is zero on}\, \ \{x_{i}=0\}, \, \ \text{and}
\end{align*}

\begin{align*}
m_{B}(\Omega)=\sum_{j=1}^{J}m_{B}(\Omega_{j}).
\end{align*}
\end{remark}
Hence \begin{align}\label{3.0.3.1.2}
	\dfrac{P_{A}(\Omega)}{\left[m_{B}(\Omega)\right]^{\frac{N+a-1}{N+b}}}\geq\min\left\{\dfrac{P_{A}\left(\Omega_{j}\right)}{\left[m_{B}(\Omega_{j})\right]^{\frac{N+a-1}{N+b}}}; 1\leq j\leq J\right\}:=\dfrac{P_{A}(\Omega_{j_{0}})}{\left[m_{B}(\Omega_{j_{0}})\right]^{\frac{N+a-1}{N+b}}},
\end{align}
since $a-b\leq 1$, moreover, the equality in $(\ref{3.0.3.1.2})$ can hold when $a-b=1$. After reflections regarding the $x_{i}$-axis, where $i\in\{1,\ldots,k\}$, we can assume that $\Omega_{j_{0}}\subset\mathbb{R}_{A}^{N}$, since this movement changes neither the volume $m_{B}(\Omega_{j_{0}})$ nor the perimeter $P_{A}(\Omega_{j_{0}})$.

In addition to that, given a Lipschitz bounded open set $\Omega\subset\mathbb{R}_{A}^{N}$, the  isoperimetric quotient $(\ref{3.0.3.1.1})$ of $\Omega$ may be approximated on $\mathbb{R}_{A}^{N}$, namely there exists a  sequence of smooth open sets $\left(\Omega_{\delta}\right)_{\delta>0}$ with $\overline{\Omega}_{\delta}\subset\Omega\subset\mathbb{R}_{A}^{N}$ satisfying \begin{align*}
	\dfrac{\displaystyle\int_{\partial\Omega_{\delta}}x^{A}d\mathcal{H}^{N-1}(x)}{\left[\displaystyle\int_{\Omega_{\delta}}x^{B}dx\right]^{\frac{N+a-1}{N+b}}}\to	\dfrac{\displaystyle\int_{\partial\Omega}x^{A}d\mathcal{H}^{N-1}(x)}{\left[\displaystyle\int_{\Omega}x^{B}dx\right]^{\frac{N+a-1}{N+b}}} \, \ \text{as}\, \ \delta\to 0.
\end{align*}

\begin{lemma}\label{3.0.4.0} Let $A=(a_{1},\ldots, a_{N})$ and $B=(b_{1},\ldots, b_{N})$ be two nonnegative vectors. Assume that $a-b\leq 1$, then
	\begin{align*}
	C_{A,B,N}=\inf\{\mathcal{Q}_{A,B,N}(u): u\in C_{0}^{1}(\mathbb{R}^{N})\backslash\{0\}\}
	\end{align*}
\end{lemma}
\begin{proof} Consider $\varepsilon>0$, then there exists a smooth bounded open set $\Omega$ such that $\overline{\Omega}\subset\mathbb{R}_{A}^{N}$, see remark \ref{3.0.3.1.0}, satisfying 
	\begin{align*}
	\mathcal{R}_{A,B,N}(\Omega)\leq C_{A,B,N}+\varepsilon.
	\end{align*} Applying the Lemma \ref{3.0.3.0} for the functions $\gamma(x)=x^{B}$, and $\omega(x)=x^{A}$, we then have
\begin{align*}
	C_{A,B,N}\geq\inf\{\mathcal{Q}_{A,B,N}(u): u\in C_{0}^{1}(\mathbb{R}^{N})\backslash\{0\}\}.
	\end{align*}
	
	To get the reverse inequality, without loss of generality, we may assume that $u$ is a nonnegative function. By coarea formula, we get 
	\begin{align}\label{3.0.4.1}
	\int_{\mathbb{R}^{N}} x^{A}\vert\nabla u\vert dx&=\int_{0}^{\infty}\int_{u=t}x^{A}\mathcal{H}^{N-1}(x)dt\nonumber\\
	&\geq C_{A,B,N}\int_{0}^{\infty}\left[\int_{u>t}x^{B}dx\right]^{\frac{N+a-1}{N+b}}dt.
	\end{align}
	
It follows from Minkowski's inequality for integrals and fubini's theorem that
	
	\begin{align}\label{3.0.4.2}
		\int_{\mathbb{R}^{N}}x^{B}\vert u\vert^{\frac{N+b}{N+a-1}}dx&=\int_{\mathbb{R}^{N}}x^{B}\left[\int_{0}^{\infty}\chi_{\{z>0; u(x)>z\}}(t)dt\right]^{\frac{N+b}{N+a-1}}dx\nonumber\\
		&=\int_{\mathbb{R}^{N}}x^{B}\left[\int_{0}^{\infty}\chi_{\{y\in\mathbb{R}^{N};u(y)>t\}}(x)dt\right]^{\frac{N+b}{N+a-1}}dx\nonumber\\
		&=\int_{\mathbb{R}^{N}}\left[\int_{0}^{\infty}\left(x^{B}\chi_{\{y\in\mathbb{R}^{N};u(y)>t\}}(x)\right)^{\frac{N+a-1}{N+b}}dt\right]^{\frac{N+b}{N+a-1}}dx\nonumber\\
		&\leq\left[\int_{0}^{\infty}\left(\int_{\mathbb{R}^{N}}x^{B}\chi_{\{y\in\mathbb{R}^{N};u(y)>t\}}(x)dx\right)^{\frac{N+a-1}{N+b}}dt\right]^{\frac{N+b}{N+a-1}}\nonumber\\
		&=\left[\int_{0}^{\infty}\left(\int_{u>t}x^{B}dx\right)^{\frac{N+a-1}{N+b}}dt\right]^{\frac{N+b}{N+a-1}}.
	\end{align}
	
	Hence, by $(\ref{3.0.4.1})$ and $(\ref{3.0.4.2})$, we then get
	\begin{align*}
			C_{A,B,N}\leq\dfrac{\displaystyle\int_{\mathbb{R}^{N}}\vert\nabla u\vert x^{A} dx}{\left[\displaystyle\int_{\mathbb{R}^{N}}\vert u\vert^{\frac{N+b}{N+a-1}}x^{B}dx\right]^{\frac{N+a-1}{N+b}}}.
	\end{align*}
This concludes the proof of the lemma.	
\end{proof}

\textbf{Proof of the Theorem \ref{0.0.0.9}} The part $(I)$ of the theorem follows from Lemmas \ref{3.0.1.0} and \ref{3.0.2.0}.

To prove the part $(II)$,  firstly we consider that $a-b<1$. Since the condition $(\ref{0.0.0.9.0})$ holds, we then get
\begin{align*}
	0\leq a_{i}-\frac{N+a-1}{N+b}b_{i}\leq\frac{N+a-1}{N+b}<1.
\end{align*}
Thus it follows from Theorem A and Lemma \ref{3.0.4.0} that \begin{align*}
C_{A,B,N}>0.
\end{align*}

We now assume that $a-b=1$. It follows from condition $(\ref{0.0.0.9.0})$ that
\begin{align*}
0\leq a_{i}-b_{i}\leq 1,
\end{align*}
for every $i\in\{1,\ldots, N\}$.

If $a_{i}-b_{i}<1$ for each $i\in\{1,\ldots, N\}$, then the theorem follows from Theorema A and Lemma \ref{3.0.4.0}. Otherwise, there exists $j\in\{1,\ldots, N\}$ such that $a_{j}-b_{j}=1$ and $a_{i}=b_{i}$ for every $i\in\{1,\ldots,N\}\backslash\{j\}$, then the result relies on the proof of the Theorem \ref{0.0.0.10} and Lemma \ref{3.0.4.0}.

{\section{Proof of the Theorem \ref{0.0.0.10}}}

The proof consists to show that if $a_i=b_i+1$, then $a_{i}= C_{A,B,N}$. To prove that \begin{align}\label{4.0.0.0}
a_{i}\leq C_{A,B,N}
\end{align}
we will use the Lemma \ref{3.0.4.0} and an idea contained in \cite{Castro}. 

Given $v\in C_{c}^{1}(\mathbb{R})$, $v\geq 0$, we have, integrating by parts that

\begin{align}\label{4.0.0.1}
\int_{\mathbb{R}}\vert y\vert^{b_{i}}v(y)dy&=\dfrac{1}{b_{i}+1}\int_{\mathbb{R}}\left(\vert y\vert^{b_{i}}y\right)'v(y)dy\nonumber\\
&=-\dfrac{1}{b_{i}+1}\int_{\mathbb{R}}\vert y\vert^{b_{i}}yv'(y)dy\nonumber\\
&\leq\dfrac{1}{a_{i}}\int_{\mathbb{R}}\vert y\vert^{a_{i}}\vert v'(y)\vert dy.	
\end{align}

We now apply the inequality (\ref{4.0.0.1}) to the function $v(y)=\overline{x}_{i}^{\overline{A}_{i}} u(x_{1},\ldots, x_{i-1},y, x_{i+1}, \ldots, x_{N})$ with $u\geq 0$, thus we then have
\begin{align*}
\int_{\mathbb{R}}\vert\vert y\vert^{b_{i}}\overline{x}_{i}^{\overline{A}_{i}} u(x_{1},\ldots x_{i-1},y,x_{i+1},\ldots, x_{N})\vert dy\leq\dfrac{1}{a_{i}}\int_{\mathbb{R}}\vert\vert y\vert^{a_{i}}\overline{x}_{i}^{\overline{A}_{i}}\partial_{y}\left(u(x_{1},\ldots,x_{i-1},y,x_{i+1},\ldots, x_{N})\right)\vert dy.
\end{align*}
Integrating with respect to the variables $x_{1},\ldots x_{i-1},x_{i+1},\ldots, x_{N}$, we obtain that
\begin{align}\label{4.0.0.2}
a_{i}\leq\dfrac{\displaystyle\int_{\mathbb{R}^{N}}\vert\nabla u(x)\vert x^{A} dx}{\displaystyle\int_{\mathbb{R}^{N}}\vert u(x)\vert x^{B} dx}.
\end{align}
Therefore, the inequality $(\ref{4.0.0.0})$ follows from Lemma \ref{3.0.4.0} and inequality $(\ref{4.0.0.2})$.

To prove the reverse inequality, we will use the proof of Lemma \ref{3.0.2.0}. Indeed, by the proof of Lemma \ref{3.0.2.0}, we get 

\begin{align*}
\dfrac{P_{A}(\Omega_{\varepsilon})}{\left[m_{B}(\Omega_{\varepsilon})\right]^{\frac{N+a-1}{N+b}}}&\leq\varepsilon^{a_{i}-\frac{N+a-1}{N+b}(b_{i}+1)}\dfrac{\frac{1}{(1+\varepsilon^{2})^{\frac{N+a-2}{2}}}\displaystyle\int_{B_{N-1}^{+}(1)}\overline{x}_{i}^{\overline{A}_{i}}\vert\overline{x}_{i}\vert^{a_{i}}d\overline{x}_{i}}{\left[\dfrac{1}{(b_{i}+1)(1+\varepsilon^{2})^{\frac{N+b}{2}}}\displaystyle\int_{B_{N-1}^{+}(1)}\overline{x}_{i}^{\overline{B}_{i}}\vert \overline{x}_{i}\vert^{b_{i}+1}d\overline{x}_{i}\right]^{\frac{N+a-1}{N+b}}}\\
&+O(\varepsilon^{a_{i}+1-\frac{N+a-1}{N+b}(b_{i}+1)})
+O(\varepsilon^{a_{i}+2-\frac{N+a-1}{N+b}(b_{i}+1)})\\
&=\dfrac{\frac{1}{(1+\varepsilon^{2})^{\frac{N+a-2}{2}}}\displaystyle\int_{B_{N-1}^{+}(1)}\overline{x}_{i}^{\overline{A}_{i}}\vert\overline{x}_{i}\vert^{a_{i}}d\overline{x}_{i}}{\dfrac{1}{(b_{i}+1)(1+\varepsilon^{2})^{\frac{N+b}{2}}}\displaystyle\int_{B_{N-1}^{+}(1)}\overline{x}_{i}^{\overline{B}_{i}}\vert\overline{x}_{i}\vert^{b_{i}+1}d\overline{x}_{i}}+O(\varepsilon)+O(\varepsilon^{2})\\
&=(b_{i}+1)(1+\varepsilon^{2})^{\frac{3}{2}}+O(\varepsilon)+O(\varepsilon^{2}),
\end{align*}
where $\Omega_{\varepsilon}$ is the same set as defined in Lemma \ref{3.0.2.0}. 
Therefore, \begin{align*}
\lim_{\varepsilon\to 0}\dfrac{P_{A}(\Omega_{\varepsilon})}{m_{B}(\Omega_{\varepsilon})}=a_{i}.
\end{align*} 
Which concludes the proof. \qed
\begin{remark} We consider again $A=(a_{1},\ldots, a_{N})$ and $B=(b_{1},\ldots, b_{N})$ two nonnegative vectors in $\mathbb{R}^{N}$. The case when $A=B$ was studied by Cabr\'{e} and Ros-Oton, and they proved that \begin{align*}
	\dfrac{P_{A}(B_{1}^{A})}{\left[m_{A}(B_{1}^{A})\right]^{\frac{N+a-1}{N+a}}}=C_{A,A,N},
	\end{align*}
where  $B_{1}^{A}:=B_{1}(0)\cap \mathbb{R}_{A}^{N}$.

The study on the existence of sets $\Omega$ in $\mathbb{R}^{N}$ that minimize the isoperimetric quotient 	
\begin{align*}
		\dfrac{P_{A}(\Omega)}{\left[m_{B}(\Omega)\right]^{\frac{N+a-1}{N+b}}}
	\end{align*}
is in preparation.
\end{remark}

\end{document}